\documentclass[11pt]{article}

\usepackage{amsfonts}
\usepackage{amsthm}
\usepackage{amsmath}
\usepackage{graphicx}
\usepackage{enumerate}

\newtheorem{theorem}{Theorem}

\newtheorem{definition}{Definition}
\newtheorem{example}{Example}
\newtheorem{lemma}{Lemma}

\numberwithin{equation}{section}
\numberwithin{definition}{section}
\numberwithin{lemma}{section}
\numberwithin{theorem}{section}
\begin{document}

\author{Sezin Aykurt Sepet\\
\small K{\i}r\c{s}ehir Ahi Evran University, Department of Mathematics, K{\i}r\c{s}ehir, Turkey\\
\small sezinaykurt@hotmail.com}
\title{Bi-slant $\xi^{\perp}$-Riemannian submersions}
\maketitle

\begin{abstract}
 We introduce bi-slant $\xi^{\perp}$-Riemannian submersions from Sasakian manifolds onto Riemannian manifolds as a generalization of slant and semi-slant $\xi^{\perp}$-Riemannian submersion. We give an example and investigate the geometry foliations. After we obtain necessary and sufficient conditions related to totally geodesicness of submersion. Finally we give decomposition theorems for total manifold of such submersions.\\


\textbf {2010 Subject Classification}: {53C15, 53C40}

\end{abstract}

\section {Introduction}

\label{intro}

The differential geometry of slant submanifolds has been studied by many authors since B.Y Chen defined slant immersions in complex geometry as a natural generalization of both holomorphic immersions and totally real immersions. Carriazo \cite{carriazo} has introduced bi-slant immersions. Then Uddin et al. \cite{uddin} have studied warped product bi-slant immersions in Kaehler manifolds. As a generalization of CR-submanifolds, slant and semi-slant submanifolds, Cabrerizo et al. \cite{cabrerizo} have defined bi-slant submanifolds of almost contact metric manifolds. Recently, Alqahtani et al. \cite{Alqahtani} have investigated warped product bi-slant submanifolds of cosymplectic manifolds.

 On the other hand Riemannian submersions was introduced by B. O'Neill \cite{oneill} and A. Gray \cite{gray}. Since then Riemannian submersions have been studied extensively by many geometers. In \cite{watson}, B. Watson defined almost Hermitian submersions between almost Hermitian manifolds. In this study, he investigated some geometric properties between base manifold and total manifolds as well as fibers.

B. Sahin \cite{sahinslant} defined slant submersions from almost Hermitian manifolds onto Riemannian manifolds as follows: Let $F$ be a Riemannian submersion from an almost Hermitian manifold $(M,g,J)$ onto a Riemannian manifold $(N,g')$. If for any nonzero vector $X\in\Gamma\left(\ker F_{*}\right)$ the angle $\theta\left(X\right)$ between $JX$ and the space $\ker F_{*}$ is a constant, i.e. it is independent of the choice of the point $p\in M$ and choice of the tangent vector $X$ in $\ker F_{*}$, then we say that $F$ is a slant submersion. In this case, the angle $\theta$ is called the slant angle of the slant submersion. Many interesting studies on several types of submersions have been done. For instance, slant and semi-slant submersions \cite{gunduzalp,gunduzalpsemi,kupeli,prasad}, anti-invariant riemannian submersions \cite{sahinanti}, semi-invariant submersions \cite{ozdemir,sahinsemi}, pointwise slant submersions \cite{sezin,lee2}, hemi-slant submersions \cite{hakan}, generic submersions \cite{Tastan3}.

Furthermore J.W. Lee \cite{lee} defined anti-invariant $\xi^{\perp}$-Riemannian submersions from almost contact metric manifolds. Later as a generalization of anti-invariant $\xi^{\perp}$-Riemannian submersions, Akyol et al studied the geometry of semi-invariant $\xi^{\perp}$-Riemannian submersion, semi-slant $\xi^{\perp}$-Riemannian submersions and conformal anti-invariant $\xi^{\perp}$-submersions from almost contact metric manifolds \cite{akyol3,akyol2,akyol4}.

The paper is organized as follows. In Section 2, we recall the basic formulas and notions needed for this paper. In Section 3 we define bi-slant $\xi^{\perp}$-Riemannian submersions from Sasakian manifold and give an example. We also investigate the geometry of leaves of distributions and find necessary and sufficient conditions for such submersions to be totally geodesic and harmonic, respectively.\\

\section { Preliminaries}
An almost contact structure $(\phi,\xi,\eta)$ on a manifold $M$ of dimension $2n+1$ is defined by a tensor field $\phi$ of type $(1,1)$, a vector field $\xi$ (Reeb vector field) and a 1-form $\eta$ satisfying
\begin{align}
\phi^{2}=-I+\eta\otimes\xi, \  \eta(\xi)=1,  \  \eta\circ\phi=0,  \ \phi\xi=0
\end{align}
where $I$ is the identity endomorphism of $TM$. There always exist a Riemannian metric $g$ on $M$ satisfying the following compatibility condition with the almost contact structure $(\phi,\xi,\eta)$
\begin{align}
g\left(\phi X,\phi Y\right)=g\left(X,Y\right)-\eta(X)\eta(Y)
\end{align}
where $X,Y$ are any vector fields on $M$. Then the manifold $M$ together with the structure $(\phi,\xi,\eta,g)$ is called an almost contact metric manifold. An almost contact metric manifold is said to be normal if
\begin{align}
\left[\phi,\phi\right]+2d\eta\otimes\xi=0
\end{align}
where $\left[\phi,\phi\right]$ is Nijenhuis tensor of $\phi$. Let $\Phi$ denote the 2-form on an almost contact metric manifold $(M,\phi,\xi,\eta,g)$ given by $\Phi(X,Y)=g\left(X,\phi Y\right)$ for any $X,Y\in \Gamma\left(TM\right)$. The $\Phi$ is called the fundamental 2-form of $M$. An almost contact metric manifold $(M,\phi,\xi,\eta,g)$ is said to be a contact metric manifold if $\Phi=d\eta$. A normal contact metric manifold is called a Sasakian manifold. Then the structure equations of Sasakian manifold are given by
\begin{align*}
\left(\nabla_{X}\phi\right)Y=g\left(X,Y\right)\xi-\eta(Y)X  \ \ \text{and}  \ \ \nabla_{X}\xi=-\phi X,
\end{align*}
where $\nabla$ is the Levi-Civita connection of $g$ and $X,Y\in\Gamma\left(TM\right)$.\\

Let $(M,g)$ and $(N,g')$ be an Riemannian manifolds with $m$ and $n$ dimension respectively $m>n$. A surjective mapping $F:M\longrightarrow N$ is said to be a Riemannian submersion if
\begin{enumerate}[i)]
\item $F$ has maximal rank
\item The differential $F_{*}$ preserves the lenghts of the horizontal vectors.
\end{enumerate}

For any $q\in N$, $F^{-1}(q)$ is an $m-n$ dimensional submanifold of $M$ called fiber. If a vector field on $M$ is always tangent to fibers then it is called vertical. If a vector field on $M$ is always orthogonal to fibers then it is called horizontal\cite{oneill}. A vector field $X$ on $M$ is called basic if it is horizontal and $F$-related to a vector field $X_{*}$ on $N$, i.e. $F_{*}X_{p}=X_{*F(p)}$ for all $p\in M$. We denote the projection morphisms on the distributions $(kerF_{*})$ and $(kerF_{*})^{\perp}$ by $\mathcal{V}$ and $\mathcal{H}$ ,respectively \cite{falcitelli}.
\\
A Riemannian submersion $F:M\longrightarrow N$ characterized by two fundamental tensor fields $\mathcal{T}$ and $\mathcal{A}$ on $M$ is defined by the following formulae
\setlength\arraycolsep{2pt}
\begin{eqnarray}
\mathcal{T}(E,F)&=&\mathcal{T}_{E}F=\mathcal{H}\nabla_{\mathcal{V}E}\mathcal{V}F+\mathcal{V}\nabla_{\mathcal{V}E}\mathcal{H}F \\
\mathcal{A}(E,F)&=&\mathcal{A}_{E}F=\mathcal{V}\nabla_{\mathcal{H}E}\mathcal{H}F+\mathcal{H}\nabla_{\mathcal{H}E}\mathcal{V}F
\end{eqnarray}
for any vector fields $E$ and $F$ on $M$ where $\nabla$ is the Levi-Civita connection of $(M,g)$.
In addition, for $X,Y\in\Gamma\left(\left(kerF_{*}\right)^{\perp}\right)$ and $U,W\in\Gamma\left(kerF_{*}\right)$ the tensor fields satisfy
\setlength\arraycolsep{2pt}
\begin{eqnarray}
\mathcal{T}_{U}W&=&\mathcal{T}_{W}U\\
\mathcal{A}_{X}Y&=&-\mathcal{A}_{Y}X= \frac{1}{2}\mathcal{V}[X,Y].
\end{eqnarray}

On the other hand, note that a Riemannian submersion $F:M\longrightarrow N$ has totally geodesic fibers if and only if $\mathcal{T}$ vanishes identically.

Now, we recall the following lemma from \cite{oneill}.
\begin{lemma}
Let $F:M\longrightarrow N$ be a Riemannian submersion between Riemannian manifolds. If $X$ and $Y$ are basic vector fields of $M$ then
\begin{enumerate}[i)]
\item $g(X,Y)=g'(X_{*},Y_{*})\circ F$,
\item the horizontal part $[X,Y]^{\mathcal{H}}$ of $[X,Y]$ is a basic vector field and\\ $F_{*}\left([X,Y]^\mathcal{H}\right)=[X_{*},Y_{*}]$,
\item $[V,X]$ is vertical for any vector field $V$ of $(kerF_{*})$,
\item $\left(\nabla^{M}_{X}Y\right)^{\mathcal{H}}$ is the basic vector field corresponding to $\nabla^{N}_{X_{*}}Y_{*}$,
\end{enumerate}
where $\nabla^{M}$ and $\nabla^{N}$ are the Levi-Civita connections on $M$ and $N$, respectively.
\end{lemma}
On the other hand from (2.2) and (2.3) we have
\setlength\arraycolsep{2pt}
\begin{eqnarray}
\nabla_{U}V &=&\mathcal{T}_{U}V+\bar{\nabla}_{U}V \\
\nabla_{U}X&=&\mathcal{H}\nabla_{U}X+\mathcal{T}_{U}X \\
\nabla_{X}U&=&\mathcal{A}_{X}U+\mathcal{V}\nabla_{X}U\\
\nabla_{X}Y&=&\mathcal{H}\nabla_{X}Y+\mathcal{A}_{X}Y
\end{eqnarray}
for $X,Y\in\Gamma\left(\left(\ker F_{*}\right)^{\perp}\right)$ and $U,V\in\Gamma\left(\ker F_{*}\right)$, where $\bar{\nabla}_{U}V=\mathcal{V}\nabla_{U}V$. Moreover, if $X$ is basic then $\mathcal{H}\nabla_{U}X=\mathcal{A}_{X}U$.

Let $(M,g)$ and $(N,g')$ be Riemannian manifolds and $\psi:M\longrightarrow N$ is a smooth mapping. The second fundamental form of $\psi$ is given by
\begin{equation}
\nabla\psi_{*}(X,Y)=\nabla^{\psi}_{X}\psi_{*}(Y)-\psi_{*}\left(\nabla^{M}_{X}Y\right)
\end{equation}
for $X,Y\in\Gamma\left(TM\right)$, where $\nabla^\psi$ is the pullback connection. Recall that $\psi$ is said to be harmonic if $trace\nabla\psi_{*}=0$ and $\psi$ is called a totally geodesic map if $\left(\nabla\psi_{*}\right)\left(X,Y\right)=0$ for $X,Y\in \Gamma\left(TM\right)$ \cite{baird}.\\

\section {Bi-Slant Submersions }

\begin{definition}
	Let $\left(M,\phi,\xi,\eta,g\right)$ be a Sasakian manifold  and $\left(N,g'\right)$ a Riemannian manifold. A Riemannian submersion $F:M\longrightarrow N$ is called a bi-slant submersion if
	
	\begin{enumerate}[i)]
		\item for nonzero any $X\in(D_{1})_{p}$ and $p\in M$, the angles $\theta_{1}$ between $\phi X$ and the space $(D_{1})_{p}$ are constant,
		\item for nonzero any $Y\in (D_{2})_{q}$ and $q\in M$, the angles $\theta_{2}$ between $\phi Y$ and the space $(D_{2})_{q}$ are constant, $F_{*}\left([X,Y]^\mathcal{H}\right)=[X_{*},Y_{*}]$,
		\item $\phi D_{1}\perp D_{2}$ and $\phi D_{2}\perp D_{1}$
	\end{enumerate}
such that $\ker F_{*}=D_{1}\oplus D_{2}$. $F$ is called proper if its bi-slant angles satisfy $\theta_{1},\theta_{2}\neq 0,\frac{\pi}{2}$.
\end{definition}

Now we give some examples of bi-slant submersions.
Suppose that $\mathbb{R}^{2n+1}$ denote a Sasakian manifold with the structure $\left(\phi,\xi,\eta,g\right)$ defined as
\begin{align*}
\phi\left(\sum_{i=1}^{n}\left(X_{i}\frac{\partial}{\partial x^{i}}+Y_{i}\frac{\partial}{\partial y^{i}}\right)+Z\frac{\partial}{\partial z}\right)&=\sum_{i=1}^{n}\left(Y_{i}\frac{\partial}{\partial x^{i}}-X_{i}\frac{\partial}{\partial y^{i}}\right),\\
\eta=\frac{1}{2}\left(dz-\sum_{i=1}^{n}y^{i}dx^{i}\right), \  \xi&=2\frac{\partial}{\partial z}\\
g=\eta\otimes\eta+\frac{1}{4}\sum_{i=1}^{n}\left(dx^{i}\otimes dx^{i}+\right.&\left.dy^{i}\otimes dy^{i}\right),
\end{align*}
where $(x^{1},...,x^{n},y^{1},...,y^{n},z)$	are the Cartesian coordinates.

\begin{example}
Let $F:\mathbb{R}^{9}\longrightarrow\mathbb{R}^{5}$ be a submersion defined by
\begin{eqnarray*}
\resizebox{.96 \textwidth}{!} {$F(x_{1},x_{2},x_{3},x_{4},y_{1},y_{2},y_{3},y_{4},z)=\left(\cos\alpha x_{1}-\sin\alpha x_{2},\frac{x_{3}+x_{4}}{\sqrt{2}},\sin\beta y_{1}+\cos\beta y_{2}, y_{3}, z\right)$}
\end{eqnarray*}
then
\begin{eqnarray*}
\ker F_{*}=span\left\{V_{1}=\sin\alpha\frac{\partial}{\partial x_{1}}+\cos\alpha\frac{\partial}{\partial x_{2}},V_{2}=\frac{1}{\sqrt{2}}\left(\frac{\partial}{\partial x_{3} }-\frac{\partial}{\partial x_{4}}\right),\right.  \\
\left.V_{3}=\left(\cos \beta \frac{\partial}{\partial y_{1}}-\sin \beta \frac{\partial}{\partial y_{2}}\right),V_{4}=\frac{\partial}{\partial y_{4} }\right\}
\end {eqnarray*}
and
\begin{eqnarray*}
\left(\ker F_{*}\right)^\perp=span\left\{H_{1}=\cos\alpha\frac{\partial}{\partial x_{1}}-\sin\alpha\frac{\partial}{\partial x_{2} }, H_{2}=\frac{1}{\sqrt{2}}\left(\frac{\partial}{\partial x_{3} }+\frac{\partial}{\partial x_{4}}\right)\right.\\
\left.H_{3}=\left(\sin\beta \frac{\partial}{\partial y_{1}}+\cos\beta \frac{\partial}{\partial y_{2}}\right),
H_{4}=\frac{\partial}{\partial y_{3}}, \xi=\frac{\partial}{\partial z}\right\}
\end {eqnarray*}

Thus we obtain $D_{1}=\left\{V_{1},V_{3}\right\}$ and $D_{2}=\left\{V_{2},V_{4}\right\}$ with the angle $\theta_{1}=\beta-\alpha$ and $\theta_{2}=\frac{\pi}{4}$. So $F$ is a bi-slant submersion.
\end{example}

Let $F$ be a bi-slant submersion from Sasakian $\left(M,\phi,\xi,\eta,g\right)$  onto a Riemannian manifold $(N,g')$. Then for $U\in\Gamma\left(\ker F_{*}\right)$, we have
\begin{equation}
U=PU+QU
\end{equation}
where $PU\in\Gamma\left(D_{1}\right)$ and $QU\in\Gamma\left(D_{2}\right)$.\\
In addition, for $U\in\Gamma\left(\ker F_{*}\right)$, we get
\begin{equation}
\phi U=\varphi U+\omega U
\end{equation}
where $\phi U\in\Gamma\left(\ker F_{*}\right)$ and $\omega U\in\Gamma\left(\ker F_{*}\right)^\perp$.\\
Similarly, for $X\in\Gamma\left(\ker F_{*}\right)^\perp$, we can write
\begin{equation}
\phi X=BX+CX
\end{equation}
where $BX\in\Gamma\left(\ker F_{*}\right)$ and $CX\in\Gamma\left(\ker F_{*}\right)^\perp$.\\
The horizontal distribution $(\ker F_{*})^{\perp}$ is decompesed as
\begin{align}
(\ker F_{*})^{\perp}=\omega D_{1}\oplus\omega D_{2}\oplus\mu
\end{align}
where $\mu$ is the complementary distribution to $\omega D_{1}\oplus\omega D_{2}$ in $(\ker F_{*})^{\perp}$ and contains $\xi$. Also it is invariant distribution $\left(\ker F_{*}\right)^{\perp}$.\\
From (3.1), (3.2) and (3.3) we arrive following equations
\begin{align}
\varphi D_{1}=D_{1}, \ \ \varphi D_{2}=D_{2}, \ \ B\omega D_{1}=D_{1}, \ \ B\omega D_{2}=D_{2}.
\end{align}
Now we can give the following theorem by using Definition 3.1 and the equation (3.2).
\begin{theorem}
Let $F$ be a Riemannian submersion from a Sasakian manifold $\left(M,\phi,\xi,\eta,g\right)$ onto a Riemannian manifold $(N,g')$. Then $F$ is a bi-slant submersion if and only if there exist bi-slant angle $\theta_{i}$ defined on $D_{i}$ such that
\begin{align*}
\varphi^{2}=-\left(\cos^{2}\theta_{i}\right)I, \  i=1,2
\end{align*}
\end{theorem}
\begin{proof}
The proof of this theorem is the similar to semi-slant submanifolds. \cite{cabrerizo}.
\end{proof}

\begin{theorem}
Let $F$ be a bi-slant submersion from a Sasakian manifold $\left(M,\phi,\xi,\eta,g\right)$ onto a Riemannian manifold $(N,g')$ with bi-slant angles $\theta_{1},\theta_{2}$.
\end{theorem}
	\begin{enumerate}[i)]
	\item $D_{1}$ is integrable if and only if
	\begin{align*}
	g\left(T_{U}\omega\varphi V+T_{V}\omega\varphi U, W\right)=&g\left(T_{U}\omega V+T_{V}\omega U,\varphi W\right)\\&+g\left(\mathcal{H}\nabla_{U}\omega V+\mathcal{H}\nabla_{V}\omega U,\omega W\right)
	\end{align*}
	\item $D_{2}$ is integrable if and only if
	\begin{align*}
	g\left(T_{W}\omega\varphi Z+T_{Z}\omega\varphi W, U\right)=&g\left(T_{W}\omega Z+T_{Z}\omega W,\varphi U\right)\\&+g\left(\mathcal{H}\nabla_{W}\omega Z+\mathcal{H}\nabla_{Z}\omega W,\omega U\right)
	\end{align*}
	for $U,V\in\Gamma\left(D_{1}\right)$ and $W,Z\in\Gamma\left(D_{2}\right)$.
\end{enumerate}
\begin{proof}
	For $U,V\in\Gamma\left(D_{1}\right)$ and $X\in\Gamma\left(\left(\ker F_{*}\right)^{\perp}\right)$, since $g\left([U,V],X\right)=0$, it is enough to show $g\left([U,V],Z\right)=0$ for $Z\in\Gamma\left(D_{2}\right)$. Then since $M$ is a Sasakian manifold we get
	\begin{align*}
	g\left([U,V],W\right)=&-g\left(\nabla_{U}\phi\psi V,W\right)+g\left(\nabla_{U}\omega V,\phi W\right)
	\\&+g\left(\nabla_{V}\phi\psi U,W\right)-g\left(\nabla_{V}\omega U,\phi W\right).
	\end{align*}
Theorem 3.1 and the equation (2.9) imply that
\begin{align*}
\sin^{2}\theta_{1} g\left([U,V],W\right)=&-g\left(T_{U}\omega\varphi V+T_{V}\omega\varphi U, W\right)+g\left(T_{U}\omega V+T_{V}\omega U,\varphi W\right)\\&+g\left(\mathcal{H}\nabla_{U}\omega V+\mathcal{H}\nabla_{V}\omega U,\omega W\right)	
\end{align*}
Similarly for $W,Z\in\Gamma\left(D_{2}\right)$ and $U\in\Gamma\left(D_{1}\right)$ it can be shown that
\begin{align*}
\sin^{2}\theta_{2} g\left([W,Z],U\right)=&-g\left(T_{W}\omega\varphi Z+T_{Z}\omega\varphi W, U\right)+g\left(T_{W}\omega Z+T_{Z}\omega W,\varphi U\right)\\&+g\left(\mathcal{H}\nabla_{W}\omega Z+\mathcal{H}\nabla_{Z}\omega W,\omega U\right).
\end{align*}
which proves (i). The proof of (ii) can be found in a similar way.
\end{proof}

\begin{theorem}
	Let $F$ be a bi-slant submersion from a Sasakian manifold $\left(M,\phi,\xi,\eta,g\right)$ onto a Riemannian manifold $(N,g')$ with bi-slant angles $\theta_{1},\theta_{2}$. Then $\left(\ker F_{*}\right)^{\perp}$ is integrable if and only if
	\begin{align*}
	g\left(A_{Y}BX-A_{X}BY,\omega U\right)=&g\left(\mathcal{H}\nabla_{X}CY-\mathcal{H}\nabla_{Y}CX,\omega U\right)+\eta(Y)g\left(Y,\omega U\right)\\&-\eta(X)g\left(Y,\omega U\right)-g\left([X,Y],\omega\varphi U\right)	
	\end{align*}
for $X,Y\in\Gamma\left(\left(\ker F_{*}\right)^{\perp}\right)$ and $U\in\Gamma\left(\ker F_{*}\right)$.
\end{theorem}
\begin{proof}
	For $X,Y\in\Gamma\left(\left(\ker F_{*}\right)^{\perp}\right)$ and $U\in\Gamma\left(\ker F_{*}\right)$. Then since $M$ is a Sasakian manifold we get
\begin{align*}
g\left([X,Y],U\right)=&-g\left(\nabla_{X}Y,\phi\varphi U\right)+g\left(\phi\nabla_{X}\omega Y,\omega U\right)
\\&+g\left(\nabla_{Y}X,\phi\varphi U\right)-g\left(\phi\nabla_{Y}\omega X,\omega U\right).
\end{align*}
From Theorem 3.1 we deduce that
\begin{align*}
\sin^{2}\theta_{1} g\left([X,Y],U\right)=&\left(\cos^{2}\theta_{2}-\cos^{2}\theta_{1}\right)g\left([X,Y],QU\right)-g\left(\nabla_{X}Y,\omega\varphi U\right)\\&+g\left(\nabla_{Y}X,\omega\varphi U\right)+g\left(\nabla_{X}\phi Y,\omega U\right)+\eta(Y)g\left(X,\omega U\right)\\&-g\left(\nabla_{Y}\phi X,\omega U\right)-\eta(X)g\left(Y,\omega U\right)
\end{align*}
Then from the equation (2.11), we have
\begin{align*}
\sin^{2}\theta_{1} g\left([X,Y],U\right)=&\left(\cos^{2}\theta_{2}-\cos^{2}\theta_{1}\right)g\left([X,Y],QU\right)-g\left([X,Y],\omega\varphi U\right)\\&+g\left(A_{X}BY,\omega U\right)+g\left(\mathcal{H}\nabla_{X}CY,\omega U\right)-g\left(A_{Y}BX,\omega U\right)\\&-g\left(\mathcal{H}\nabla_{Y}CX,\omega U\right)-\eta(X)g\left(Y,\omega U\right)+\eta(Y)g\left(X,\omega U\right)
\end{align*}
which gives the desired equation.
\end{proof}

\begin{theorem}
Let $F$ be a bi-slant submersion from a Sasakian manifold $\left(M,\phi,\xi,\eta,g\right)$ onto a Riemannian manifold $(N,g')$ with bi-slant angles $\theta_{1},\theta_{2}$. Then the distribution $D_{1}$ defines a totally geodesic foliation if and only if
\begin{eqnarray*}
 -g\left(T_{U}\omega\varphi V,W\right)=g\left(T_{U}\omega V,\varphi W\right)+g\left(\mathcal{H}\nabla_{U}\omega V,\omega W\right)
\end{eqnarray*}
and
\begin{align*}
g\left(T_{U}\omega V,BX\right)=g\left(\mathcal{H}\nabla_{U}\omega\varphi V,X\right)-g\left(\mathcal{H}\nabla_{U}\omega V,CX\right)
\end{align*}
where $U,V \in D_{1}$, $W\in D_{2}$ and $X\in \Gamma\left(\left(\ker F_{*}\right)^{\perp}\right)$.
\end{theorem}

\begin{proof}
From the equations (2.1), (2.2) and (3.2) for any $U,V \in D_{1} $ and $W \in D_{2}$ we can write
\setlength\arraycolsep{2pt}
\begin{eqnarray*}
g\left(\nabla_{U}V, W\right)&=&-g\left(\phi\nabla_{U}V,\phi W\right)\\
&=&-g\left(\phi\nabla_{U} \psi V, W\right)+g\left(\nabla_{U}\omega V,\phi W\right)
\end{eqnarray*}
Then Theorem 3.1 implies that
\begin{eqnarray}
\sin^2\theta_{1}g\left(\nabla_{U}V,W\right)&=&-g\left(\nabla_{U}\omega\varphi V, W\right)+g\left(\nabla_{U}\omega V,\phi W\right) \nonumber
\end{eqnarray}
Hence by using the equation (2.9) we have
\begin{eqnarray*}
\sin^2\theta_{1}g\left(\nabla_{U}V,W\right)&=&g\left(\mathcal{H}\nabla_{U}\omega V,\omega W\right)+g\left(\mathcal{T}_{U}\omega V,\varphi W\right) \nonumber \\
&-&g\left(\mathcal{T}_{U}\omega\varphi V,W\right).
\end{eqnarray*}
which proves the first equation.
On the other hand, for $X\in \Gamma\left(\left(\ker F_{*}\right)^{\perp}\right)$, we derive
\begin{align*}
g\left(\nabla_{U}V,X\right)=&g\left(\nabla_{U}\phi V,\phi X\right)+g\left(V,\phi U\right)\eta(X)\\
=&-g\left(\phi\nabla_{U}\psi V,X\right)+g\left(\nabla_{U}\omega V,\phi X\right)+g\left(V,\phi U\right)\eta(X).
\end{align*}
Considering Theorem 3.1 we arrive at
\begin{align*}
\sin^{2}\theta_{1}g\left(\nabla_{U}V,X\right)=-g\left(\nabla_{U}\omega\varphi V,X\right)+g\left(\nabla_{U}\omega V,\phi X\right).
\end{align*}
From (2.9) we have
\begin{align*}
\sin^{2}\theta_{1}g\left(\nabla_{U}V,X\right)=&-g\left(\mathcal{H}\nabla_{U}\omega\psi V,X\right)+g\left(\mathcal{H}\nabla_{U}\omega V,CX\right)+g\left(T_{U}\omega V,BX\right)
\end{align*}
which gives the second equation.
\end{proof}

\begin{theorem}
Let $F$ be a bi-slant submersion from a Sasakian manifold $\left(M,\phi,\xi,\eta,g\right)$ onto a Riemannian manifold $(N,g')$ with bi-slant angles $\theta_{1},\theta_{2}$. Then the distribution $D_{2}$ defines a totally geodesic foliation if and only if
\begin{eqnarray*}
-g\left(T_{W}\omega\varphi Z,U\right)=g\left(T_{W}\omega Z,\varphi U\right)+g\left(\mathcal{H}\nabla_{W}\omega Z,\omega U\right)
\end{eqnarray*}
and
\begin{align*}
g\left(T_{W}\omega Z,BX\right)=g\left(\mathcal{H}\nabla_{W}\omega\varphi Z,X\right)-g\left(\mathcal{H}\nabla_{W}\omega Z,CX\right)
\end{align*}
where $U\in D_{1}$, $W,Z\in D_{2}$ and $X\in \Gamma\left(\left(\ker F_{*}\right)^{\perp}\right)$.
\end{theorem}

\begin{proof}
By using similar method in Theorem 3.4 the proof of this theorem can be easily made.
\end{proof}

\begin{theorem}
Let $F$ be a bi-slant submersion from a Sasakian manifold $\left(M,\phi,\xi,\eta,g\right)$ onto a Riemannian manifold $(N,g')$ with bi-slant angles $\theta_{1},\theta_{2}$. Then the distribution $\left(\ker F_{*}\right)^\perp$ defines a totally geodesic foliation on $M$ if and only if
\begin{eqnarray*}
\left(\cos^{2}\theta_{1}-\cos^{2}\theta_{2}\right)g\left(A_{X}Y,QU\right)&=&-g\left(\mathcal{H}\nabla_{X}Y,\omega\varphi U\right)+g\left(\omega\mathcal{A}_{X}BY,\omega U\right) \nonumber\\
&+&g\left(\mathcal{H}\nabla_{X}CY,\omega U\right)+\eta(Y)\left(X,\omega U\right)
\end{eqnarray*}
where $X,Y \in\Gamma\left(\ker F_{*}\right)^\perp$ and $U\in\left(\ker F_{*}\right)$.
\end{theorem}

\begin{proof}
For $X,Y \in\Gamma\left(\ker F_{*}\right)^\perp$ and $U\in\left(\ker F_{*}\right)$ we can write
\setlength\arraycolsep{2pt}
\begin{eqnarray*}
g\left(\nabla_{X}Y,U\right)&=&g\left(\phi\nabla_{X}Y,\phi U\right)\\
&=&-g\left(\nabla_{X}Y,\phi\varphi U\right)+g\left(\phi\nabla_{X}Y,\omega U\right)
\end{eqnarray*}
By using Theorem 3.1 we obtain
\begin{align*}
g\left(\nabla_{X}Y,U\right)=&\cos^2\theta_{1}g\left(\nabla_{X}Y,PU\right)+\cos^2\theta_{2}g\left(\nabla_{X}Y,QU\right)-g\left(\nabla_{X}Y,\omega\varphi U\right)\\&+g\left(\phi\nabla_{X}Y,\omega U\right)
\end{align*}
From the equations (2.11) and $PU=U-QU$ we have
\begin{align*}
\sin^{2}\theta_{1}g\left(\nabla_{X}Y,U\right)=&\left(\cos^{2}\theta_{2}-\cos^{2}\theta_{1}\right)g\left(A_{X}Y,QU\right)\\
&-g\left(\mathcal{H}\nabla_{X}Y,\omega\varphi U\right)+g\left(\mathcal{A}_{X}BY,\omega U\right)\\
&+g\left(\mathcal{H}\nabla_{X}CY,\omega U\right)+\eta\left(Y\right)g\left(X,\omega U\right)
\end{align*}
Thus we have the equation (3.10).
\end{proof}

\begin{theorem}
Let $F$ be a bi-slant submersion from a Sasakian manifold $\left(M,\phi,\xi,\eta,g\right)$ onto a Riemannian manifold $(N,g')$ with bi-slant angles $\theta_{1},\theta_{2}$. Then the distribution $\left(\ker F_{*}\right)$ defines a totally geodesic foliation on $M$ if and only if
\begin{align*}
\left(\cos^{2}\theta_{1}-\cos^{2}\theta_{2}\right)g\left(\mathcal{T}_{U}QV,X\right)=&-g\left(\mathcal{H}\nabla_{U}\omega\varphi V,X\right)+g\left(\mathcal{T}_{U}\omega V,BX\right)\nonumber\\
&+g\left(\mathcal{H}\nabla_{U}\omega V,CX\right)
\end{align*}
where $X\in\Gamma\left(\ker F_{*}\right)^\perp$ and $U,V\in\left(\ker F_{*}\right)$.
\end{theorem}

\begin{proof}
Let suppose that $X \in\Gamma\left(\ker F_{*}\right)^\perp$ and $U,V\in\left(\ker F_{*}\right)$. Then we get
\setlength\arraycolsep{2pt}
\begin{align*}
g\left(\nabla_{U}V,X\right)=&g\left(\nabla_{U}PV,X\right)+g\left(\nabla_{U}QV,X\right)\\
=&g\left(\phi\nabla_{U}PV,\phi X\right)+g\left(\phi U,PV\right)\eta(X)+g\left(\phi\nabla_{U}QV,\phi X\right)\\
&+g\left(\phi U,QV\right)\eta(X)
\end{align*}
Considering that $M$ is a Sasakian manifold we arrive
\begin{align*}
g\left(\nabla_{U}V,X\right)=&-g\left(\nabla_{U}\varphi^{2} PV,X\right)-g\left(\nabla_{U}\varphi^{2} QV,X\right)-g\left(\nabla_{U}\omega\varphi V,X\right)\\&+g\left(\nabla_{U}\omega V,\phi X\right)
\end{align*}
From (2.8), (2.9) we obtain
\begin{align*}
\sin^{2}\theta_{1}g\left(\nabla_{U}V,X\right)=&\left(\cos^{2}\theta_{2}-\cos^{2}\theta_{1}\right)g\left(\mathcal{T}_{U}QV,X\right)\\
&-g\left(\mathcal{H}\nabla_{U}\omega\varphi V,X\right)+g\left(\mathcal{H}\nabla_{U}\omega V,CX\right)\\
&+g\left(\mathcal{T}_{U}\omega V,BX\right)
\end{align*}
which shows our assertion.
\end{proof}

\begin{theorem}
Let $F$ be a bi-slant submersion from a Sasakian manifold $\left(M,\phi,\xi,\eta,g\right)$ onto a Riemannian manifold $(N,g')$ with bi-slant angles $\theta_{1},\theta_{2}$. Then $F$ is totally geodesic if and only if
\begin{align*}
\left(\cos^{2}\theta_{2}-\cos^{2}\theta_{1}\right)g\left(A_{X}QU,Y\right)=&g\left(\mathcal{H}\nabla_{X}\omega U,Y\right)+g\left(\mathcal{H}\nabla_{X}\omega\varphi U,Y\right)\nonumber\\
&+g\left(\omega\mathcal{A}_{X}\omega U,Y\right)+g\left(U,\phi X\right)\eta(Y)
\end{align*}
and
\begin{align*}
\left(\cos^{2}\theta_{1}-\cos^{2}\theta_{2}\right)g\left(\mathcal{T}_{U}QV,X\right)=&-g\left(\mathcal{H}\nabla_{U}\omega\varphi V,X\right)+g\left(\mathcal{T}_{U}\omega V,BX\right)\nonumber\\
&+g\left(\mathcal{H}\nabla_{U}\omega V,CX\right)
\end{align*}
where $X,Y\in\Gamma\left(\ker F_{*}\right)^\perp$ and $U,V\in\left(\ker F_{*}\right)$.
\end{theorem}

\begin{proof}
Firstly since $F$ is a Riemannian submersion for $X,Y\in\Gamma\left(\ker F_{*}\right)^\perp$ we have
\begin{equation*}
\left(\nabla F_{*}\right)\left(X,Y\right)=0.
\end{equation*}
Therefore for $X,Y\in\Gamma\left(\ker F_{*}\right)^\perp$ and $U,V\in\left(\ker F_{*}\right)$ it is enough to show that $\left(\nabla F_{*}\right)\left(U,V\right)=0$ and $\left(\nabla F_{*}\right)\left(X,U\right)=0$. So we can write
\begin{align*}
g'\left(\left(\nabla F_{*}\right)(X,U),F_{*}Y\right)=-g'\left(F_{*}\left(\nabla_{X}U\right),F_{*}Y\right)=-g\left(\nabla_{X}U,Y\right).
\end{align*}
Then using the equation (2.6) and (2.7), we obtain the first equation of Theorem 3.8
\setlength\arraycolsep{2pt}
\begin{align*}
g\left(\nabla_{X}U,Y\right)=-g\left(\nabla_{X}\phi\varphi U,Y\right)+g\left(\nabla_{X}\omega U,\phi Y\right)-g\left(U,\phi X\right)\eta(Y)
\end{align*}
From  the equations (2.6), (2.7) and Theorem 3.1 we find
\begin{align*}
\sin^{2}\theta_{1}g\left(\nabla_{X}U,Y\right)=&\left(\cos^{2}\theta_{2}-\cos^{2}\theta_{1}\right)g\left(A_{X}QU,Y\right)-g\left(\mathcal{H}\nabla_{X}\omega U,Y\right)\\&-g\left(\mathcal{H}\nabla_{X}\omega\varphi U,Y\right)-g\left(\omega A_{X}\omega U,Y\right)-g\left(U,\phi X\right)\eta(Y)
\end{align*}
Also, for get the second equation of Theorem 3.8 we have
\begin{equation*}
g'\left(\left(\nabla F_{*}\right)(U,V),F_{*}\right)=-g\left(\nabla_{U}V,X\right).
\end{equation*}
Then using the equation (2.8) and (2.9), we arrive
\setlength\arraycolsep{2pt}
\begin{align*}
g\left(\nabla_{U}V,X\right)=&\left(\cos^{2}\theta_{2}-\cos^{2}\theta_{1}\right)g\left(\mathcal{T}_{U}QV,X\right)-g\left(\mathcal{H}\nabla_{U}\omega\varphi V,X\right)\\&+g\left(\mathcal{T}_{U}\omega V,BX\right)+g\left(\mathcal{H}\nabla_{U}\omega V,CX\right)
\end{align*}
which completes proof.
\end{proof}

\section {Decompositions Theorems}
In this section we give decompositions theorems using the existence of bi-slant $\xi^{\perp}$-Riemannian submersion. We assume that $g$ is a Riemannian metric tensor on the manifold $M=M_{1}\times M_{2}$ and the canonical foliations $D_{M_{1}}$ and $D_{M_{2}}$ intersect vertically everywhere. Then $g$ is the metric tensor of a usual product of Riemannian manifold if and only if $D_{M_{1}}$ and $D_{M_{2}}$ are totally geodesic foliations.

Now we can write the following theorems by using Theorem 3.4-3.7,
\begin{theorem}
 Let $F$ be a bi-slant submersion from a Sasakian manifold $\left(M,\phi,\xi,\eta,g\right)$ onto a Riemannian manifold $(N,g')$ with bi-slant angles $\theta_{1},\theta_{2}$. Then $M$ is a locally product manifold of the form $M_{D_{1}}\times M_{D_{2}}\times M_{\left(\ker F_{*}\right)^{\perp}}$ if and only if
 \begin{align*}
 -g\left(T_{U}\omega\varphi V,W\right)=&g\left(T_{U}\omega V,\varphi W\right)+g\left(\mathcal{H}\nabla_{U}\omega V,\omega W\right),\\
g\left(T_{U}\omega V,BX\right)=&g\left(\mathcal{H}\nabla_{U}\omega\varphi V,X\right)-g\left(\mathcal{H}\nabla_{U}\omega V,CX\right),\\
-g\left(T_{W}\omega\varphi Z,U\right)=&g\left(T_{W}\omega Z,\varphi U\right)+g\left(\mathcal{H}\nabla_{W}\omega Z,\omega U\right),\\
g\left(T_{W}\omega Z,BX\right)=&g\left(\mathcal{H}\nabla_{W}\omega\varphi Z,X\right)-g\left(\mathcal{H}\nabla_{W}\omega Z,CX\right)
\end{align*}
 and
 \begin{eqnarray*}
\left(\cos^{2}\theta_{1}-\cos^{2}\theta_{2}\right)g\left(A_{X}Y,QU\right)&=&-g\left(\mathcal{H}\nabla_{X}Y,\omega\varphi U\right)+g\left(\omega\mathcal{A}_{X}BY,\omega U\right) \nonumber\\
&+&g\left(\mathcal{H}\nabla_{X}CY,\omega U\right)+\eta(Y)\left(X,\omega U\right)
\end{eqnarray*}
 for $U,V\in\Gamma(D_{1})$, $W,Z\in\Gamma(D_{2})$ and $X,Y\in\Gamma\left(\left(\ker F_{*}\right)^{\perp}\right)$.
\end{theorem}

\begin{theorem}
	Let $F$ be a bi-slant submersion from a Sasakian manifold $\left(M,\phi,\xi,\eta,g\right)$ onto a Riemannian manifold $(N,g')$ with bi-slant angles $\theta_{1},\theta_{2}$. Then $M$ is a locally product manifold of the form $M_{\ker F_{*}}\times M_{\left(\ker F_{*}\right)^{\perp}}$ if and only if
\begin{align*}
\left(\cos^{2}\theta_{1}-\cos^{2}\theta_{2}\right)g\left(\mathcal{T}_{U}QV,X\right)=&-g\left(\mathcal{H}\nabla_{U}\omega\varphi V,X\right)+g\left(\mathcal{T}_{U}\omega V,BX\right)\nonumber\\
&+g\left(\mathcal{H}\nabla_{U}\omega V,CX\right)
\end{align*}
	and
\begin{align*}
\left(\cos^{2}\theta_{1}-\cos^{2}\theta_{2}\right)g\left(A_{X}Y,QU\right)=&-g\left(\mathcal{H}\nabla_{X}Y,\omega\varphi U\right)+g\left(\omega\mathcal{A}_{X}BY,\omega U\right) \nonumber\\
&+g\left(\mathcal{H}\nabla_{X}CY,\omega U\right)+\eta(Y)\left(X,\omega U\right)
\end{align*}
	for $U,V\in\Gamma(D_{1})$, $W,Z\in\Gamma(D_{2})$ and $X,Y\in\Gamma\left(\left(\ker F_{*}\right)^{\perp}\right)$.
\end{theorem}

\begin{theorem}
	Let $F$ be a bi-slant submersion from a Sasakian manifold $\left(M,\phi,\xi,\eta,g\right)$ onto a Riemannian manifold $(N,g')$ with bi-slant angles $\theta_{1},\theta_{2}$ such that $\left(\ker F_{*}\right)^{\perp}=\omega D_{1}\oplus\omega D_{2}\oplus\langle\xi\rangle$. Then $M$ is a locally product manifold of the form $M_{D_{1}}\times M_{D_{2}}\times M_{\left(\ker F_{*}\right)^{\perp}}$ if and only if
	\begin{align*}
	-g\left(T_{U}\omega\varphi V,W\right)=&g\left(T_{U}\omega V,\varphi W\right)+g\left(\mathcal{H}\nabla_{U}\omega V,\omega W\right),\\
	g\left(T_{U}\omega V,\phi X\right)=&g\left(\mathcal{H}\nabla_{U}\omega\varphi V,X\right),\\
	-g\left(T_{W}\omega\varphi Z,U\right)=&g\left(T_{W}\omega Z,\varphi U\right)+g\left(\mathcal{H}\nabla_{W}\omega Z,\omega U\right),\\
	g\left(T_{W}\omega Z,\phi X\right)=&g\left(\mathcal{H}\nabla_{W}\omega\varphi Z,X\right)
	\end{align*}
	and
	\begin{align*}
		\left(\cos^{2}\theta_{1}-\cos^{2}\theta_{2}\right)g\left(A_{X}Y,QU\right)=&-g\left(\mathcal{H}\nabla_{X}Y,\omega\varphi U\right)+g\left(\omega\mathcal{A}_{X}\phi Y,\omega U\right)\\
		&+\eta(Y)\left(X,\omega U\right)
	\end{align*}
	for $U,V\in\Gamma(D_{1})$, $W,Z\in\Gamma(D_{2})$ and $X,Y\in\Gamma\left(\left(\ker F_{*}\right)^{\perp}\right)$.
\end{theorem}

\begin{theorem}
	Let $F$ be a bi-slant submersion from a Sasakian manifold $\left(M,\phi,\xi,\eta,g\right)$ onto a Riemannian manifold $(N,g')$ with bi-slant angles $\theta_{1},\theta_{2}$ such that $\left(\ker F_{*}\right)^{\perp}=\omega D_{1}\oplus\omega D_{2}\oplus\langle\xi\rangle$. Then $M$ is a locally product manifold of the form $M_{\ker F_{*}}\times M_{\left(\ker F_{*}\right)^{\perp}}$ if and only if
	\begin{align*}
	\left(\cos^{2}\theta_{1}-\cos^{2}\theta_{2}\right)g\left(\mathcal{T}_{U}QV,X\right)=&-g\left(\mathcal{H}\nabla_{U}\omega\varphi V,X\right)+g\left(\mathcal{T}_{U}\omega V,\phi X\right)\nonumber
	\end{align*}
	and
	\begin{align*}
	\left(\cos^{2}\theta_{1}-\cos^{2}\theta_{2}\right)g\left(A_{X}Y,QU\right)=&-g\left(\mathcal{H}\nabla_{X}Y,\omega\varphi U\right)+g\left(\omega\mathcal{A}_{X}\phi Y,\omega U\right) \nonumber\\
	&+\eta(Y)\left(X,\omega U\right)
	\end{align*}
	for $U,V\in\Gamma(D_{1})$, $W,Z\in\Gamma(D_{2})$ and $X,Y\in\Gamma\left(\left(\ker F_{*}\right)^{\perp}\right)$.
\end{theorem}

\end{document}